\documentclass[preprint,10pt]{elsarticle}
\usepackage{amssymb}
\usepackage{atbegshi}
\AtBeginDocument{\AtBeginShipoutNext{\AtBeginShipoutDiscard}}
\usepackage[english,french]{babel}
\usepackage[utf8]{inputenc} 
\usepackage{times}
\usepackage{subcaption}
\usepackage{amsthm}
\usepackage{lmodern}
\usepackage[a4paper]{geometry}
\usepackage{epstopdf}
\usepackage{mathrsfs}
\usepackage{latexsym,amssymb,amsfonts}
\usepackage{enumerate}
\usepackage{graphicx}
 \usepackage{url}

\usepackage{float}
\usepackage{graphicx}

\usepackage{alltt}

\usepackage{amsmath}
\usepackage[labelfont=bf]{caption}
\usepackage{hyperref}
\hypersetup{
    colorlinks,
    citecolor=red,
    filecolor=red,
    linkcolor=red,
    urlcolor=red
}

\newtheorem{thm}{Theorem}[section]

\newtheorem{lem}[thm]{Lemma}

\theoremstyle{definition}

\def\R{\mathbb{R}}

\def\P{\mathbb P}

\def\E{\mathbb E}

\journal{.}
\usepackage{geometry}
\begin{document}
\selectlanguage{english}
\begin{frontmatter}

\title{Developing new techniques for obtaining the threshold of a stochastic SIR epidemic model with $3$-dimensional Lévy process}
\author{Driss Kiouach\footnote{Corresponding author.\\
E-mail addresses: \href{d.kiouach@uiz.ac.ma}{d.kiouach@uiz.ac.ma} (D. Kiouach), \href{yassine.sabbar@edu.uiz.ac.ma}{yassine.sabbar@edu.uiz.ac.ma} (Y. Sabbar)} and  Yassine Sabbar }
\address{LPAIS Laboratory, Faculty of Sciences Dhar El Mahraz, Sidi Mohamed Ben Abdellah University, Fez, Morocco.}   
\vspace*{1cm}

\begin{abstract}
This paper considers the classical SIR epidemic model driven by a multidimensional Lévy jump process. We consecrate to develop a mathematical method to obtain the asymptotic properties of the perturbed model. Our method differs from previous approaches by the use of the comparison theorem, mutually exclusive possibilities lemma, and some new techniques of the stochastic differential systems. In this framework,  we derive the threshold which can determine the existence of a unique ergodic stationary distribution or the extinction of the epidemic. Numerical simulations about different perturbations are realized to confirm the obtained theoretical results. \vskip 2mm


\end{abstract}
\end{frontmatter}
\
\nocite{2009ProcDETAp}
\section{Introduction}
The stochastic systems are largely used in order to describe and control the dissemination of diseases into a population \cite{1}. It will continue to be one of the vigorous themes in mathematical biology due to its significance \cite{2}. The stochastic SIR epidemic model with mass action rate is a standard model among many mathematical models that present the first tentative to understand the random transmission mechanisms of infectious epidemics \cite{3}. Taking the stochastic disturbances into account, the traditional perturbed SIR epidemic model is described by the following model:
\begin{flalign}
&\begin{cases}
dS(t)=\big(A-\mu_1 S(t)-\beta S(t)I(t)\big)dt+S(t)dH_1(t),\\
dI(t)=\big(\beta S(t)I(t) -(\mu_2+\gamma)I(t)\big)dt+I(t)dH_2(t),\\
dR(t)=\big(\gamma I(t)-\mu_1 R(t)\big)dt+R(t)dH_3(t),
\end{cases}&
\label{s1}
\end{flalign}
where $H(t)=(H_1(t),H_2(t),H_3(t))$ is a $3$-dimensional  stochastic process modeling the intensity of random perturbations of the system. $S(t)$ denotes the number of individuals sensitive to the disease, $I(t)$ denotes the number of contagious individuals and $R(t)$ denotes the number of recovered individuals with full immunity. The positive parameters of the perturbed model (\ref{s1}) are given in the table \ref{t1}. Before explaining the aim of our contribution, we first present the following cases:
\begin{enumerate}
\item Case 1: $H(t)=0$. The system (\ref{s1}) becomes deterministic which is the object of extensive
studies. The equilibrium of (\ref{s1}) is characterized by the basic reproduction number $\mathcal{R}_0=\frac{\beta A}{\mu_1(\mu_2+\gamma)}$ which is the threshold between the persistence and the extinction of a disease \cite{5}. If $\mathcal{R}_0\leq 1$, then the system (\ref{s1}) has only the disease-free equilibrium $P^0$ which is globally asymptotically stable; this means that the disease will extinct. If $\mathcal{R}_0 > 1$, $P^0$ will become unstable, therefore there exists a globally asymptotically stable equilibrium $P^*$; this means that the disease will persist.
\begin{center}
\begin{tabular}{ll}
\hline
Parameters \hspace*{0.5cm} & Interpretation \\
\hline
$A$ & The recruitment rate corresponding to births and immigration.  \\
$\mu_1$ & The natural mortality  rate. \\
$\beta$ & The transmission rate from infected to susceptible individuals. \\
$\gamma$ & The rate of recovering. \\
$\mu_2=\mu_1+\alpha$ & The general mortality rate, where $\alpha>0$ is the disease-related death rate. \\
\hline
\end{tabular}
\captionof{table}{Biological meanings of the parameters in model (\ref{s1}).}
\label{t1}
\end{center}
\item Case 2: $H_i(t)=\sigma_i W_i(t)$, $(i=1,2,3)$ where  $W_i(t)$ $(i=1,2,3)$ are independent standard Brownian motions and $\sigma_i$ $(i=1,2,3)$ are the intensities of environmental white noises \cite{new0}. There are numerous significant works that analyzed the dynamics of the model (\ref{s1}) with white noises. For instance:
\begin{enumerate}
\item In \cite{54}, the authors investigate the asymptotic behavior of the model (\ref{s1}) around the disease-free equilibrium of the deterministic model.
\item In \cite{55}, the authors analyze the long-time behavior of the stochastic SIR epidemic model (\ref{s1}). Precisely, they discussed the convergence of densities of the solution in $L^1$.
\end{enumerate}
\item Case 3: $H_i(t)=\sigma_iW_i(t)+\int^t_0\int_Z\eta_i(u) \widetilde{N}(dt,du)$ where $W_{i}(t)$ $(i = 1, 2,3)$ are independent Brownian motions and $\sigma_{i}>0$ $(i = 1, 2,3)$ are their intensities. $N$ is a Poisson counting measure with compensating martingale $\widetilde{N}$ and characteristic measure $\nu$ on a measurable subset $Z$ of $(0,\infty)$ satisfying $\nu(Z)<\infty$. $W_{i}(t)$ $(i = 1, 2,3)$ are independent of $N$. It assumed that $\nu$ is a Lévy measure such that $\widetilde{N}(dt,du)=N(dt,du)-\nu(du)dt$. The bounded function $\eta_i:\; Z\times\Omega\to\R$ is $\mathfrak{B}(Z)\times\mathcal{F}_t$-measurable and  continuous with respect to $\nu$. Our work considers the Lévy jumps process case and treats the following model:
\begin{flalign}
&\begin{cases}
dS(t)=\big(A-\mu_1 S(t)-\beta S(t)I(t)\big)dt+\sigma_1 S(t) dW_1(t)+\int_Z \eta_1(u)S(t^{-})\widetilde{N}(dt,du),\\
dI(t)=\big(\beta S(t)I(t) -(\mu_2+\gamma)I(t)\big)dt+\sigma_2 I(t) dW_2(t)+\int_Z \eta_2(u)I(t^{-})\widetilde{N}(dt,du),\\
dR(t)=\big(\gamma I(t)-\mu_1 R(t)\big)dt+\sigma_3 R(t) dW_3(t)+\int_Z \eta_3(u)R(t^{-})\widetilde{N}(dt,du),
\end{cases}&
\label{s2}
\end{flalign}
where $S(t^{-})$, $I(t^{-})$ and $R(t^{-})$ are the left limits of $S(t)$, $I(t)$ and $R(t)$, respectively. The jumps process used to model some unexpected and severe environmental disturbances (tsunami, floods, earthquakes, hurricanes, whirlwinds, etc.) on the disease outbreak.
\end{enumerate}
The previous contributions on the dynamic behavior of the model (\ref{s2}) can be summarised as follows:
\begin{enumerate}
\item In \cite{56}, the authors examined how the Lévy noise influences the behavior around the equilibriums.  More precisely, they investigated the asymptotic behavior of the model (\ref{s2}) around the disease-free equilibrium $P^0$ of the deterministic model as well as the dynamics around the endemic equilibrium $P^{*}$.
\item In \cite{57}, the authors proved that the parameter
\begin{align*}
\mathcal{T}^s_0=\Big(\mu_2+\gamma\Big)^{-1}\left(\frac{\beta A}{\mu_1}-\frac{\sigma_2^2}{2}-\int_Z \eta_2(u)-\ln(1+\eta_2(u))\nu(du)\right)
\end{align*}
is the threshold of the stochastic model (\ref{s2}). More specifically, if $\mathcal{T}^s_0<1$, the epidemic eventually vanishes with probability one; while if $\mathcal{T}^s_0>1$, the disease persists almost surely.
\end{enumerate}
As far as we know, no previous research has investigated the ergodicity of the stochastic system (\ref{s2}). It is of interest to study the long term behavior of the stochastic epidemic model (\ref{s2}) which provides a link between mathematical study, actual diseases, and public health planning. Our contribution aims to develop a mathematical method to study the ergodicity of the model (\ref{s2}) as an important asymptotic property which means that the stochastic model has a unique stationary distribution that predicts the survival of the infected population in the future. Moreover, this work focuses on solving the problem overlooked by many researchers. For instance, in \cite{58}, the authors used the existence of the stationary distribution of an auxiliary stochastic differential equation for establishing the threshold expression of the stochastic chemostat model with Lévy jumps. However, the obtained threshold still unknown due to the ignorance of the explicit form of the existed stationary distribution. Without using the stationary distribution of the auxiliary process, we will exploit new techniques in order to obtain the explicit form of the threshold which can close the gap left by using the classical method. Further, we employe the Feller property,  the mutually exclusive possibilities lemma and the stochastic comparison theorem to prove that $\mathcal{T}^s_0$ is the threshold between the existence of the ergodic stationary distribution and the extinction. It should be noted that the approach used to prove the ergodicity is different from the Khasminskii method widely used in the literature (see for example \cite{new1,new2,new3}), and the method used to prove the  extinction is different from that used in \cite{57}.\\

Our work is organized as follows. In section \ref{sec1}, we show that there exists a unique global positive solution to the system (\ref{s2}) with any positive initial value. Under suitable assumptions, the threshold of the stochastic model is obtained in section \ref{sec2}. One example is provided to demonstrate our analytical results in section \ref{sec3}. Finally, a conclusion is presented to end this paper.
\section{Well-posedness of the stochastic model (\ref{s2})} \label{sec1}
For the purpose of well analyzing our model (\ref{s2}), it necessary that we make the following standard assumptions:\\

\begin{itemize}
\item \textbf{($A_1$)} We assume that for a given $K>0$, there exists a constant $L_{K}>0$ such that
\begin{align*}
\int_Z |F_i(x,u)-F_i(y,u)|^2\nu(du)<L_{K}|x-y|^2,\hspace*{0.1cm}\forall \;|x|\vee|y|\leq K,
\end{align*}
where $F_i(x,u)=x\eta_i(u)$ $(i=1,2,3)$.
\item \textbf{($A_2$)} $\forall u\in Z$, we assume that $1+\eta_i(u)>0$, $(i=1,2,3)$ and $\int_Z\big( \eta_i(u)-\ln (1+\eta_i(u))\big)\nu(du)<\infty$.
\item \textbf{($A_3$)} We suppose that exists a constant $\kappa>0$, such that $\int_Z \big(\ln(1+\eta_i(u))\big)^2\nu(du)\leq\kappa<\infty$.
\item \textbf{($A_4$)} We assume that for some $p\geq\frac{1}{2}$, $\chi_2=\mu_1-\frac{(2p-1)}{2}\max\{\sigma_1^2,\sigma_2^2\}-\frac{1}{2p}\ell$, where $$\ell=\int_Z\big((1+\eta_1(u)\vee \eta_2(u))^{2p}-1-\eta_1(u)\wedge\eta_2(u)\big)\nu(du)<\infty.$$
\end{itemize}
By the assumption \textbf{($A_1$)}, the coefficients of the system (\ref{s2}) are locally Lipschitz continuous, then for any initial value $(S(0), I(0), R(0))\in \R^{3}_{+}$ there is a unique local solution $(S(t), I(t), R(t))$ on $[0,\tau_{e})$, where $\tau_{e}$ is the explosion time. In the following theorem, our goal is to show that the solution is positive and global. 
\begin{thm}
For any initial value $(S(0),I(0),R(0))\in \R^3_{+}$, there exists a unique positive solution $(S(t),I(t),R(t))$ of the system (\ref{s2}) on $t\geq 0$, and the solution will stay in $\R^{3}_{+}$ almost surely. 
\label{thmp}
\end{thm}
\begin{proof}
 We prove that $\tau_{e}=\infty$ a.s. Let $\epsilon_0>0$ be sufficiently large, such that $S(0)$, $I(0)$, $R(0)$ lie within the interval $\big[\frac{1}{\epsilon_0},\epsilon_0\big]$. For each integer $\epsilon\geq \epsilon_0$, we define the following stopping time:
\begin{align*}
\tau_{\epsilon}&=\inf \left\lbrace t\in [0,\tau_{e}) :  \min\{S(t),I(t),R(t)\}\leq \frac{1}{\epsilon} \;\; \mbox{or}\;\; \max\{S(t),I(t),R(t)\}\geq\epsilon \right\rbrace.
\end{align*}
Evidently, $\tau_{\epsilon}$ is increasing as $\epsilon\to\infty$. Set $\tau_{\infty}=\underset{\epsilon\to\infty}{\lim}\tau_{\epsilon}$ whence $\tau_{\infty}\leq \tau_{e}$. If we can prove that $\tau_{\infty}=\infty$ a.s., then $\tau_{e}=\infty$ and the solution $(S(t),I(t),R(t))\in \R^{3}_{+}$ for all $t\geq 0$ almost surely. Specifically, we need to prove that $\tau_{\infty}=\infty$ a.s. Suppose the opposite, then there is a pair of positive constants $T>0$ and $k\in(0,1)$ such that $\P\{\tau_{\infty}\leq T\}>k$.
Hence, there is an integer $\epsilon_1\geq \epsilon_0$ such that
\begin{align}
\P\{\tau_{\epsilon}\leq T\}\geq k\hspace{0.3cm}\mbox{for all}\hspace{0.3cm} \epsilon\geq \epsilon_1.
\label{11}
\end{align}
Define a $C^2$-function $V:\R^3_+\to \R_+$ by
\begin{align*}
V(S,I,R)=\left(S-m-m\ln \frac{S}{m}\right)+(I-1-\ln I)+(R-1-\ln R),
\end{align*}
where $\alpha>0$ is a positive constant to be determined later. Obviously, this function is nonnegative which can be seen from $x-1-\ln x >0$ for $x> 0$. \\For $0\leq t\leq \tau_{\epsilon}\wedge T$, using Itô's formula, we obtain that
\begin{align*}
dV(S,I,R)&=LV(S,I,R)dt+\left(1-\frac{m}{S}\right)\sigma_1 S dW_1(t)+\left(1-\frac{1}{I}\right)\sigma_2 I dW_2(t)\\&\;\;\;+\left(1-\frac{1}{R}\right)\sigma_3 R dW_3(t)+\int_Z\begin{Bmatrix} \eta_1(u)S(t^{-})-m \ln(1+\eta_1(u))\\+\eta_2(u)I(t^{-})- \ln(1+\eta_2(u))\\+ \eta_3(u)R(t^{-})- \ln(1+\eta_3(u))&\end{Bmatrix}\widetilde{N}(dt,du),
\end{align*}
where,
\begin{align*}
LV(S,I,R)&=A-\mu_1 S-\frac{m A}{S}+m \beta I+m\mu_1-(\mu_2+\gamma)I-\beta S +(\mu_2+\gamma)+\gamma I-\mu_1 R\\&\;\;\;-\gamma\frac{I}{R}+\mu_1+\frac{m\sigma_1^2}{2}+\frac{\sigma_2^2}{2}+\frac{\sigma_3^2}{2}+\int_Z\begin{Bmatrix} m\eta_1(u)- m\ln(1+\eta_1(u))\\+\eta_2(u)- \ln(1+\eta_2(u))\\+ \eta_3(u)- \ln(1+\eta_3(u))&  \end{Bmatrix}\nu(du).
\end{align*}
Then
\begin{align*}
LV(S,I,R)&\leq A-\mu_2I+m\beta I+\mu_1+m\mu_1+\mu_2+\gamma+\frac{m\sigma_1^2}{2}+\frac{\sigma_2^2}{2}+\frac{\sigma_3^2}{2}\\&\;\;\;+\int_Z\begin{Bmatrix} m\eta_1(u)- m\ln(1+\eta_1(u))\\+\eta_2(u)- \ln(1+\eta_2(u))\\+ \eta_3(u)- \ln(1+\eta_3(u))&\end{Bmatrix}\nu(du).
\end{align*}
Given the fact that $x-\ln(1+x)\geq0$ for all $x> 1$ and the hypothesis $(A_2)$, we define
\begin{align*}
J_1&=\int_Z\begin{Bmatrix} m\eta_1(u)- m\ln(1+\eta_1(u))\\+\eta_2(u)- \ln(1+\eta_2(u))\\+ \eta_3(u)- \ln(1+\eta_3(u))&\end{Bmatrix}\nu(du).
\end{align*}
To simplify, we choose $m=\frac{\mu_2}{\beta}$, then we obtain
\begin{align*}
LV(S,I,R)\leq A-\mu_2I+m\beta I+\mu_1+m\mu_1+\mu_2+\gamma+\frac{m\sigma_1^2}{2}+\frac{\sigma_2^2}{2}+\frac{\sigma_3^2}{2}+J_1
\equiv J_2.
\end{align*}
Therefore,
\begin{align*}
\int^{\tau_{\epsilon}\wedge T}_0dV(S(t),I(t),R(t))&\leq \int^{\tau_{\epsilon}\wedge T}_0 J_2 dt+\int^{\tau_{\epsilon}\wedge T}_0\int_Z\begin{Bmatrix} \eta_1(u)S(t^{-})-m \ln(1+\eta_1(u))\\+\eta_2(u)I(t^{-})- \ln(1+\eta_2(u))\\+ \eta_3(u)R(t^{-})- \ln(1+\eta_3(u))&\end{Bmatrix}\widetilde{N}(dt,du).
\end{align*}
Taking expectation yields
\begin{align*}
\E V(S(\tau_\epsilon \wedge T), I(\tau_\epsilon \wedge T), R(\tau_\epsilon \wedge T))\leq V(S(0), I(0), R(0))+J_2 T.
\end{align*}
Setting $\Omega_\epsilon=\{\tau_\epsilon\leq T\}$ for $\epsilon\geq \epsilon_0$ and by (\ref{11}), $\P(\Omega_\epsilon)\geq k$. For $\omega\in\Omega_\epsilon$, there is some component of $S(\tau_\epsilon)$, $I(\tau_\epsilon)$ and $R(\tau_\epsilon)$ equals either $\epsilon$ or $\frac{1}{\epsilon}$. Hence, $V(S(\tau_\epsilon),I(\tau_\epsilon),R(\tau_\epsilon))$ is not less than $\epsilon-1-\ln\epsilon$ or $\frac{1}{\epsilon}-1-\ln\frac{1}{\epsilon}$. That is 
\begin{align*}
V(S(\tau_\epsilon ), I(\tau_\epsilon ),R(\tau_\epsilon ))\geq (\epsilon-1-\ln \epsilon)\wedge\left(\frac{1}{\epsilon}-1-\ln\frac{1}{\epsilon}\right).
\end{align*}
Consequently,
\begin{align*}
V(S(0), I(0), R(0))+J_2 T&\geq \E ( \mathbf{1}_{\Omega_\epsilon} V(S(\tau_\epsilon ,\omega), I(\tau_\epsilon ,\omega),R(\tau_\epsilon ,\omega)))\\
&\geq k \left((\epsilon-1-\ln \epsilon)\wedge\left(\frac{1}{\epsilon}-1-\ln\frac{1}{\epsilon}\right)\right).
\end{align*}
Extending $\epsilon$ to $\infty$ leads to the contradiction. Thus, $\tau_{\infty}= \infty$ a.s. which completes the proof of the theorem.
\end{proof}
\section{Threshold analysis of the model (\ref{s2})} \label{sec2}
The aim of the following theorem is to determine the threshold for the SDE model (\ref{s2}). 
\begin{thm}
The parameter $\mathcal{T}^s_0$ is the threshold of the stochastic model (\ref{s2}). That is to say that:
\begin{enumerate}
\item If $\mathcal{T}^s_0>1$, then the stochastic system (\ref{s2}) admits a unique stationary distribution and it has the ergodic property for any initial value $(S(0),I(0),R(0))\in\R^3_+$. 
\item If $\mathcal{T}^s_0<1$, then the epidemic dies out exponentially with probability one.
\end{enumerate}
\label{thm1}
\end{thm}
Before proving the main theorem, we prepare five useful Lemmas. Consider the following subsystem
\begin{align}
\begin{cases}d\psi(t)=(A-\mu_1\psi(t))dt+\sigma_1\psi(t)dW_1(t)+\int_Z\eta_1(u)\psi(t^{-})\tilde{\mathcal{N}}(dt,du) \hspace{0.5cm}\forall t>0\\
\psi(0)=S(0)>0.\end{cases}
\label{s4}
\end{align}
\begin{lem}\cite{166}
Let $(S(t),I(t),R(t))$ be the positive solution of the system (\ref{s2}) with any given initial condition $(S(0),I(0),R(0))\in\R^3_+$. Let also $\psi(t)\in\R_+$ be the solution of the equation (\ref{s4}) with any given initial value $\psi(0)=S(0)\in\R_+$. Then
\begin{enumerate}
\item 
\begin{align*}
\underset{t\to\infty}{\lim}\frac{\psi(t)}{t}=0, \hspace{0.3cm}\underset{t\to\infty}{\lim}\frac{S(t)}{t}=0,    \hspace{0.2cm}\mbox{and}\hspace{0.2cm}  \underset{t\to\infty}{\lim}\frac{I(t)}{t}=0 \hspace{0.5cm}\mbox{a.s.}
\end{align*}
\item \begin{align*}
&\underset{t\to\infty}{\lim}\frac{\int^t_0\int_Z \eta_1(u)\psi(s^{-})\widetilde{N}(ds,du)}{t}=0,\\
&\underset{t\to\infty}{\lim}\frac{\int^t_0\int_Z \eta_1(u)S(s^{-})\widetilde{N}(ds,du)}{t}=0,\\
&\underset{t\to\infty}{\lim}\frac{\int^t_0\int_Z \eta_2(u)I(s^{-})\widetilde{N}(ds,du)}{t}=0\hspace{0.1cm}\mbox{a.s.}
\end{align*}
\end{enumerate}
\label{lem1m}
\end{lem}
\begin{lem}
Let $\psi(t)$ be the solution of the system (\ref{s4}) with an initial value $\psi(0)\in\R_{+}$. Then, 
\begin{align*}
\underset{t\to\infty}{\lim}\frac{1}{t}\int^t_0\psi(s)ds=\frac{A}{\mu_1}\hspace{0.2cm}\mbox{a.s.}
\end{align*} 
\label{lemmas}
\end{lem}
\begin{proof}
Integrating from $0$ to $t$ on both sides of (\ref{s4}) yields
\begin{align*}
\frac{\psi(t)-\psi(0)}{t}=A-\frac{\mu_1}{t}\int^t_0\psi(s)ds+\frac{\sigma_1}{t}\int_0^t \psi(s)dW_1(s)+\frac{1}{t}\int^t_0\int_Z \eta_1(u)\psi(s^{-})\widetilde{N}(ds,du).
\end{align*}
Clearly, we can derive that
\begin{align*}
\frac{1}{t}\int^t_0\psi(s)ds=\frac{A}{\mu_1}+\frac{\sigma_1}{\mu_1 t}\int_0^t \psi(s^{-})dW_1(s)+\frac{1}{\mu_1 t}\int^t_0\int_Z \eta_1(u)\psi(s^{-})\widetilde{N}(ds,du).
\end{align*}
According to lemma \ref{lem1m} and the large number theorem for martingales, we can easily verify that
\begin{align*}
\underset{t\to\infty}{\lim}\frac{1}{t}\int^t_0\psi(s)ds=\frac{A}{\mu_1}\hspace{0.2cm}\mbox{a.s.}
\end{align*}
\end{proof}

\begin{lem}
Let $(S(t),I(t),R(t))$ be the solution of (\ref{s2}) with initial value $(S(0),I(0),R(0))\in\R^3_+$. Then 
\begin{enumerate}
\item $\E\big((S(t)+I(t))^{2p}(t)\big)\leq (S(0)+I(0))^{2p}e^{\{-p\chi_2t\}}+ \frac{2\chi_1}{\chi_2}$;
\item $\underset{t\to +\infty}{\lim \sup} \frac{1}{t}\int^t_0\E\big((S(s)+I(s))^{2p}\big)ds\leq \frac{2\chi_1}{\chi_2}$\;\; a.s.
\end{enumerate}
where $\chi_1=\underset{x>0}{\sup}\{Ax^{{2p}-1}-\frac{ \chi_2}{2}x^{2p}\}$.
\label{L1}
\end{lem}
\begin{proof}
Making use of Itô's lemma, we obtain
\begin{align*}
d(S(t)+I(t))^{2p}&\leq 2p[S(t)+I(t)]^{2p-1}\big(A-\mu_1 S(t)-(\mu_1+\alpha+\gamma)I(t)\big)dt\\&\;\;\;+p(2p-1)[S(t)+I(t)]^{2p-2}(\sigma_1^2S^2(t)+\sigma_2^2I^2(t))dt\\&\;\;\;+2p[S(t)+I(t)]^{2p-1}(\sigma_1S(t)dW_1(t)+\sigma_2I(t)dW_2(t))\\&\;\;\;+\int_Z[S(t)+I(t)]^{2p}\big((1+\eta_1(u)\vee \eta_2(u))^{2p}-1-\eta_1(u)\wedge\eta_2(u)\big)\nu(du)dt\\&\;\;\;+\int_Z[S(t)+I(t)]^{2p}\big((1+\eta_1(u)\vee \eta_2(u))^{2p}-\eta_1(u)\wedge\eta_2(u)\big)\tilde{\mathcal{N}}(dt,du).
\end{align*}
Then
\begin{align*}
&\leq 2p[S(t)+I(t)]^{2p-2}\Big\{A[S(t)+I(t)]-\Big(\mu_1-\frac{(2p-1)}{2}\max\{\sigma_1^2,\sigma_2^2\}\\&\;\;\;-\frac{1}{2p}\int_Z\big((1+\eta_1(u)\vee \eta_2(u))^{2p}-1-\eta_1(u)\wedge\eta_2(u)\big)\nu(du)\Big)[S(t)+I(t)]^{2}\Big\}\\&\;\;\;+2p[S(t)+I(t)]^{2p-1}(\sigma_1S(t)dW_1(t)+\sigma_2I(t)dW_2(t))\\&\;\;\;+\int_Z[S(t)+I(t)]^{2p}\big((1+\eta_1(u)\vee \eta_2(u))^{2p}-\eta_1(u)\wedge\eta_2(u)\big)\tilde{\mathcal{N}}(dt,du).
\end{align*}
We choose neatly $p\geq\frac{1}{2}$ such that
\begin{align*}
\chi_2&=\mu_1-\frac{(2p-1)}{2}\max\{\sigma_1^2,\sigma_2^2\}-\frac{1}{2p}\int_Z\big((1+\eta_1(u)\vee \eta_2(u))^{2p}-1-\eta_1(u)\wedge\eta_2(u)\big)\nu(du)> 0.
\end{align*}
Hence 
\begin{align*}
d(S(t)+I(t))^{2p}&\leq 2p[S(t)+I(t)]^{2p-2}\Big\{\chi_1-\frac{\chi_2}{2}[S(t)+I(t)]^{2p}\Big\}dt\\&\;\;\;
+2p[S(t)+I(t)]^{2p-1}(\sigma_1S(t)dW_1(t)+\sigma_2I(t)dW_2(t))\\&\;\;\;+\int_Z[S(t)+I(t)]^{2p}\big((1+\eta_1(u)\vee \eta_2(u))^{2p}-\eta_1(u)\wedge\eta_2(u)\big)\tilde{\mathcal{N}}(dt,du).
\end{align*}
On the other hand, we have
\begin{align*}
d(S(t)+I(t))^{2p}e^{p\chi_2t}&=p\chi_2[S(t)+I(t)]^{2p}e^{p\chi_2t}+e^{p\chi_2t}d(S(t)+I(t))^{2p}\\&\leq 2p\chi_1e^{p\chi_2t}+e^{p\chi_2t}2p(S(t)+I(t))^{2p-1}(\sigma_1S(t)dW_1(t)+\sigma_2I(t)dW_2(t))\\&\;\;\;+\int_Ze^{p\chi_2t}[S(t)+I(t)]^{2p}\big((1+\eta_1(u)\vee \eta_2(u))^{2p}-\eta_1(u)\wedge\eta_2(u)\big)\tilde{\mathcal{N}}(dt,du).
\end{align*}
Then by taking integrations and taking the expectations, we get
\begin{align*}
(S(t)+I(t))^{2p}&\leq(S(0)+I(0))^{2p}e^{-p\chi_2t}+2p\chi_1\int^t_0 e^{p\chi_2(t-s)}ds\\&\leq(S(0)+I(0))^{2p}e^{-p\chi_2t}+\frac{2\chi_1}{\chi_2}.
\end{align*}
Obviously, we obtain
\begin{align*}
\underset{t\to +\infty}{\lim \sup} \frac{1}{t}\int^t_0\E(S(t)+I(t))^{2p}(u))du\leq (S(0)+I(0))^{2p}\underset{t\to +\infty}{\lim \sup} \frac{1}{t}\int^t_0e^{-p \chi_2u}du+\frac{2\chi_1}{\chi_2}=\frac{2\chi_1}{\chi_2}.
\end{align*}
\end{proof}
\begin{lem}\cite{58}
Let $h(t)>0$, $k(t)\geq 0$ and $G(t)$ be functions on $[0,+\infty)$, $c\geq 0$ and $d>0$ be constants, such that $\underset{t\to\infty}{\lim}\frac{G(t)}{t}=0$ and 
\begin{align*}
\ln h(t)\leq ct+k(t)-d\int^t_0h(s)ds+G(t).
\end{align*}
If $k(t)$ is a non-decreasing function, then 
\begin{align*}
\underset{t\to\infty}{\lim \sup }\frac{1}{t}\bigg(-k(t)+d\int^t_0h(s)ds\bigg)\leq c.
\end{align*}
\label{nlem}
\end{lem}
\begin{lem}[\cite{10}]
Let $X(t)\in \R^n$ be a stochastic Feller process, then either an ergodic probability measure exists, or 
\begin{align}
\underset{t\to \infty}{\lim }\underset{\nu}{\sup }\frac{1}{t}\int^t_0\int \P(u,x,\Sigma)\nu(dx)du=0,\hspace{0.2cm}\mbox{for any compact set}\hspace{0.2cm}\Sigma\in\R^n,
\label{imp}
\end{align}
where the supremum is taken over all initial distributions $\nu$ on $R^d$ and $\P(t,x,\Sigma)$ is the probability for $X(t)\in \Sigma$ with $X(0)=x\in \R^n$.
\label{poss}
\end{lem}
\begin{proof}[Proof of Theorem \ref{thm1}]
Similar to the proof of Lemma 3.2. in \cite{11}, we briefly verify the Feller property of the SDE model (\ref{s2}). The main purpose of the next analysis is to prove that (\ref{imp}) is impossible. \\Applying Itô's formula gives
\begin{align}
d\ln I(t)&=\Big(\beta S(t)-(\mu_2+\gamma)-\frac{\sigma_2^2}{2}-\int_Z \eta_2(u)-\ln(1+\eta_2(u))\nu(du)\Big)dt\nonumber\\&\;\;\;+\sigma_2dW_2(t)+\int_Z\ln(1+\eta_2(u))\tilde{\mathcal{N}}(dt,du).
\label{itoln}
\end{align}
Therefore
\begin{align*}
d\Big\{\ln I(t)-\frac{\beta}{\mu_1}\big(\psi(t)-S(t)\big)\Big\}&=\Big(\beta S(t)-(\mu_2+\gamma)-\frac{\sigma_2^2}{2}-\int_Z \eta_2(u)-\ln(1+\eta_2(u))\nu(du)\Big)dt\\&\;\;\;-\frac{\beta}{\mu_1}\Big(-\mu_1(\psi(t)-S(t))+\beta S(t)I(t)\Big)dt+\sigma_2dW_2(t)\\&\;\;\;-\frac{\beta}{\mu_1}(\psi(t)-S(t))dW_1(t)+\int_Z\ln(1+\eta_2(u))\tilde{\mathcal{N}}(dt,du)\\&\;\;\;-\frac{\beta }{\mu_1}\int_Z \eta_1(u)(\psi(t)-S(t))\tilde{\mathcal{N}}(dt,du).
\end{align*}
Hence
\begin{align}
d\Big\{\ln I(t)-\frac{\beta}{\mu_1}\big(\psi(t)-S(t)\big)\Big\}\nonumber&=\Big(\beta \psi(t)-(\mu_2+\gamma)-\frac{\sigma_2^2}{2}-\int_Z \eta_2(u)-\ln(1+\eta_2(u))\nu(du)\Big)dt\nonumber\\&\;\;\;-\frac{\beta^2 S(t)I(t)}{\mu_1}dt+\sigma_2dW_2(t)-\frac{\beta}{\mu_1}(\psi(t)-S(t))dW_1(t)\nonumber\\&\;\;\;+\int_Z\ln(1+\eta_2(u))\tilde{\mathcal{N}}(dt,du)-\frac{\beta }{\mu_1}\int_Z \eta_1(u)(\psi(t)-S(t))\tilde{\mathcal{N}}(dt,du).
\label{ito}
\end{align}
Integrating from $0$ to $t$ on both sides of (\ref{ito}) yields
\begin{align*}
&\ln\frac{I(t)}{I(0)}-\frac{\beta}{\mu_1}(\psi(t)-S(t))+\frac{\beta}{\mu_1}(\psi(0)-S(0))\nonumber\\&=\int_0^t\beta \psi(s)ds-\Big((\mu_2+\gamma)+\frac{\sigma_2^2}{2}+\int_Z \eta_2(u)-\ln(1+\eta_2(u))\nu(du)\Big)\nonumber\\&\;\;\;-\frac{\beta^2}{\mu_1}\int_0^t S(s)I(s)ds+\sigma_2W_2(t)-\frac{\beta}{\mu_1}\int_0^t(\psi(s)-S(s))dW_1(s)\nonumber\\&\;\;\;+\int_0^t\int_Z\ln(1+\eta_2(u))\tilde{\mathcal{N}}(ds,du)-\frac{\beta }{\mu_1}\int_0^t\int_Z \eta_1(u)(\psi(s)-S(s))\tilde{\mathcal{N}}(ds,du).
\end{align*}
Then we have
\begin{align}
\int_0^t\beta S(s)I(s)ds&=\frac{\mu_1}{\beta}\int_0^t\beta \psi(s)ds-\frac{\mu_1}{\beta}\Big((\mu_2+\gamma)+\frac{\sigma_2^2}{2}+\int_Z \eta_2(u)-\ln(1+\eta_2(u))\nu(du)\Big)\nonumber\\&\;\;\;+(\psi(t)-S(t))-(\psi(0)-S(0))-\frac{\mu_1}{\beta}\ln\frac{I(t)}{I(0)}+\frac{\mu_1\sigma_2}{\beta}W_2(t)\nonumber\\&\;\;\;-\int_0^t(\psi(s)-S(s))dW_1(s)+\frac{\mu_1}{\beta}\int_0^t\int_Z\ln(1+\eta_2(u))\tilde{\mathcal{N}}(ds,du)\nonumber\\&\;\;\;-\int_0^t\int_Z \eta_1(u)(\psi(s)-S(s))\tilde{\mathcal{N}}(ds,du).
\label{ito1}
\end{align}
Let
\begin{align*}
M_1(t)&=\frac{\mu_1\sigma_2}{\beta}W_2(t)-\int_0^t(\psi(s)-S(s))dW_1(s)\\&\;\;\;+\frac{\mu_1}{\beta}\int_0^t\int_Z\ln(1+\eta_2(u))\tilde{\mathcal{N}}(ds,du)\\&\;\;\;-\int_0^t\int_Z \eta_1(u)(\psi(s)-S(s))\tilde{\mathcal{N}}(ds,du).
\end{align*}
We know that $ \int_0^t\int_Z\ln(1+\eta_2(u))\tilde{\mathcal{N}}(ds,du)$ is a local martingale with quadratic variation
\begin{align*}
&\Big\langle \int_0^t\int_Z\ln(1+\eta_2(u))\tilde{\mathcal{N}}(ds,du),\int_0^t\int_Z\ln(1+\eta_2(u))\tilde{\mathcal{N}}(ds,du)\Big\rangle=\Big(\int_Z\big(\ln(1+\eta_2(u)\big)^2\nu(du)\Big)t.
\end{align*}
By using the Strong Low of Large Numbers, we get $\underset{t\to +\infty}{\lim }\frac{M_1(t)}{t}=0$, a.s.\\
From the system (\ref{s2}), we obtain
\begin{align}
d(S(t)+I(t))&=\big(A-\mu_1S(t)-(\mu_2+\gamma)I(t)\big)dt+\sigma_1S(t)dW_1(t)\nonumber\\&\;\;\;+\sigma_2I(t)dW_2(t)+\int_Z\big(\eta_1(u)S(t^{-})+\eta_2(u)I(t^{-})\big)\tilde{\mathcal{N}}(ds,du).
\label{si}
\end{align}
Applying Itô's formula to the equality (\ref{si}) gives that
\begin{align*}
d\ln\bigg(\frac{1}{S(t)+I(t)}\bigg)&=\frac{-A}{S(t)+I(t)}+\frac{\mu_1S(t)+(\mu_2+\gamma)I(t)}{S(t)+I(t)}+\frac{\sigma_1^2S^2(t)+\sigma_2^2I^2(t)}{2(S(t)+I(t))^2}\\&\;\;\;-\int_Z\bigg( \ln\frac{(1+\eta_1(u))S(t)+(1+\eta_2(u))I(t)}{S(t)+I(t)}-\frac{\eta_1(u)S(t)+\eta_2I(t)}{S(t)+I(t)}\bigg)\nu(du)\\&\;\;\;-\frac{\sigma_1S(t)}{S(t)+I(t)}dW_1(t)-\frac{\sigma_2I(t)}{S(t)+I(t)}dW_2(t)\\&\;\;\;-\int_Z\ln\frac{(1+\eta_1(u))S(t)+(1+\eta_2(u))I(t)}{S(t)+I(t)}\tilde{\mathcal{N}}(dt,du).
\end{align*}
Taking integration, we get
\begin{align*}
\ln\bigg(\frac{1}{S(t)+I(t)}\bigg)&=\ln\bigg(\frac{1}{S(0)+I(0)}\bigg)-A\int_0^t\frac{1}{S(s)+I(s)}ds+M_2(t)+M_3(t),
\end{align*}
where
\begin{align*}
M_2(t)&=\int_0^t\frac{\mu_1 S(s)+(\mu_2+\gamma)I(s)}{S(s)+I(s)}ds+\int_0^t\frac{\sigma_1^2S^2(s)+\sigma_2^2I^2(s)}{2(S(s)+I(s))^2}ds\\&\;\;\;-\int_0^t\int_Z\bigg( \ln\frac{(1+\eta_1(u))S(s)+(1+\eta_2(u))I(s)}{S(s)+I(s)}-\frac{\eta_1(u)S(s)+\eta_2I(s)}{S(s)+I(s)}\bigg)\nu(du)ds,
\end{align*}
and
\begin{align*}
M_3(t)&=-\int^t_0\frac{\sigma_1S(s)}{S(s)+I(s)}dW_1(s)-\int^t_0\frac{\sigma_2I(s)}{S(s)+I(s)}dW_1(s)\\&\;\;\;-\int^t_0\int_Z\ln\frac{(1+\eta_1(u))S(s)+(1+\eta_2(u))I(s)}{S(s)+I(s)}\tilde{\mathcal{N}}(ds,du).
\end{align*}
By lemma \ref{nlem}, we get
\begin{align*}
\underset{t\to +\infty}{\lim \sup}\frac{1}{t}\bigg(\int^t_0\frac{A}{S(s)+I(s)}ds-M_2(t)\bigg)\leq 0\;\;\;\mbox{a.s.}
\end{align*}
Then
\begin{align*}
\underset{t\to +\infty}{\lim \sup}\frac{1}{t}\ln\big(S(t)+I(t)\big)\leq 0 \;\;\;\mbox{a.s.}
\end{align*}
Hence
\begin{align*}
\underset{t\to +\infty}{\lim \sup}\frac{1}{t}\ln\frac{I(t)}{I(0)}\leq \underset{t\to +\infty}{\lim \sup}\frac{1}{t}\ln\frac{\big(S(t)+I(t)\big)}{I(0)}\leq 0 \;\;\;\mbox{a.s.}
\end{align*}
Thus, it follows from (\ref{ito1}) that
\begin{align*}
&\underset{t\to +\infty}{\lim \inf} \frac{1}{t}\int^t_0 \beta S(s)I(s)du\\&\geq\frac{\mu_1}{\beta}\bigg(\underset{t\to +\infty}{\lim \inf} \frac{1}{t}\int^t_0\beta \psi(s)ds-\Big((\mu_2+\gamma)+\frac{\sigma_2^2}{2}+\int_Z \eta_2(u)-\ln(1+\eta_2(u))\nu(du)\Big)\bigg)\\&=\frac{\mu_1}{\beta}\bigg(\underset{t\to +\infty}{\lim } \frac{1}{t}\int^t_0\beta \psi(s)ds-\Big((\mu_2+\gamma)+\frac{\sigma_2^2}{2}+\int_Z \eta_2(u)-\ln(1+\eta_2(u))\nu(du)\Big)\bigg)\\
&=\frac{\mu_1}{\beta}(\mathcal{T}_0^s-1)>0\hspace{0.5cm}\mbox{a.s.}
\end{align*}
To continue our analysis, we need to set the following subsets: $\Omega_1=\{(S,I,R)\in\R^3_+|\hspace{0.1cm}S\geq \epsilon,\hspace{0.1cm}\mbox{and},\hspace{0.1cm} I\geq \epsilon\}$, $\Omega_2=\{(S,I,R)\in\R^3_+|\hspace{0.1cm}S\leq \epsilon\}$ and $\Omega_3=\{(S,I,R)\in\R^3_+|\hspace{0.1cm}I\leq \epsilon\}$ where $\epsilon>0$ is a positive constant to be determined later. Therefore, we get
\begin{align*}
&\underset{t\to +\infty}{\lim \inf}\frac{1}{t}\int^t_0\E\Big( \beta S(s)I(s)\mathbf{1}_{\Omega_1}\Big)ds\\&\geq \underset{t\to +\infty}{\lim \inf}\frac{1}{t}\int^t_0\E\Big( \beta S(s)I(s)\Big)ds-\underset{t\to +\infty}{\lim \sup}\frac{1}{t}\int^t_0\E\Big( \beta S(s)I(s)\mathbf{1}_{\Omega_2}\Big)ds-\underset{t\to +\infty}{\lim \sup}\frac{1}{t}\int^t_0\E\Big(\beta S(s)I(s)\mathbf{1}_{\Omega_3}\Big)ds\\&\geq \frac{\mu_1}{\beta}(\mathcal{T}_0^s-1)-\beta \epsilon \underset{t\to +\infty}{\lim \sup}\frac{1}{t}\int^t_0\E(I(s))ds-\beta \epsilon \underset{t\to +\infty}{\lim \sup}\frac{1}{t}\int^t_0\E(S(s))ds.
\end{align*}
By lemma \ref{L1}, we see that
\begin{align*}
\underset{t\to +\infty}{\lim \inf}\frac{1}{t}\int^t_0\E\Big( \beta S(s)I(s)\mathbf{1}_{\Omega_1}\Big)ds&\geq \frac{\mu_1}{\beta}(\mathcal{T}_0^s-1)- \frac{2A\beta \epsilon}{\mu_1-\ell}.
\end{align*}
We can choose $\epsilon \leq \frac{\mu_1}{4\beta^2A}(\mu_1-\ell)(\mathcal{T}_0^s-1)$, and then we obtain
\begin{align}
\underset{t\to +\infty}{\lim \inf}\frac{1}{t}\int^t_0\E\Big( \beta S(u)I(u)\mathbf{1}_{\Omega_1}\Big)du&\geq \frac{\mu_1}{2\beta}(\mathcal{T}_0^s-1)>0 \;\;\;\mbox{a.s.}
\label{22}
\end{align}
Let $q=a_0>1$ be a positive integer such that $1<p=\frac{a_0}{a_0-1}$, $\mu_1-\frac{(2p-1)}{2}\max\{\sigma_1^2,\sigma_2^2\}-\frac{1}{2p}\ell> 0$ and $\frac{1}{q}+\frac{1}{p}=1$. By utilizing the Young inequality $xy\leq \frac{x^p}{p}+\frac{y^q}{q}$ for all $x$,$y>0$, we get
\begin{align*}
&\underset{t\to +\infty}{\lim \inf}\frac{1}{t}\int^t_0\E\big( \beta S(u)I(u)\mathbf{1}_{\Omega_1}\big)du\\&\leq \underset{t\to +\infty}{\lim \inf}\frac{1}{t}\int^t_0\E\bigg(p^{-1}(\eta \beta S(u)I(u))^p+q^{-1}\eta^{-q}\mathbf{1}_{\Omega_1}\bigg)du\\&\leq \underset{t\to +\infty}{\lim \inf}\frac{1}{t}\int^t_0 \E\big(q^{-1}\eta^{-q}\mathbf{1}_{\Omega_1}\big)du+p^{-1}(\eta \beta)^p\underset{t\to +\infty}{\lim \sup}\frac{1}{t}\int^t_0\E\big((S(u)+I(u))^{2p}\big)du,
\end{align*}
where $\eta$ is a positive constant satisfying
\begin{align*}
\eta^p\leq\frac{p\mu_1\chi_1\beta^{-(p+1)}}{8 \chi_2}(\mathcal{T}_0^s-1).
\end{align*}
By lemma \ref{L1} and (\ref{22}), we deduce that
\begin{align}
\underset{t\to +\infty}{\lim \inf}\frac{1}{t}\int^t_0\E(\mathbf{1}_{\Omega_1})du&\geq q\eta^q\Bigg(\frac{\mu_1}{2\beta}(\mathcal{T}_0^s-1)-\frac{2\chi_2\eta^p\beta^p}{p\chi_1}\Bigg)\geq \frac{\mu_1q\eta^q}{4\beta}(\mathcal{T}_0^s-1)>0\;\;\;\mbox{a.s.}
\label{23}
\end{align}
Setting $\Omega_4=\{(S,I,R)\in\R^3_+|\hspace{0.1cm}S\geq \zeta,\hspace{0.1cm}\mbox{or},\hspace{0.1cm} I\geq \zeta\}$ and $\Sigma=\{(S,I,R)\in\R^3_+|\hspace{0.1cm}\epsilon \leq S\leq \zeta,\hspace{0.1cm}\mbox{and},\hspace{0.1cm} \epsilon \leq I\leq \zeta\}$ where $\zeta>0$ is a positive constant to be explained in the following. By using the Tchebychev inequality, we can observe that
\begin{align*}
\E(\mathbf{1}_{\Omega_4})\leq \P(S(t)\geq \zeta )+\P(I(t)\geq \zeta )\leq\frac{1}{\zeta}\E(S(t)+I(t))\leq\frac{1}{\zeta}\bigg(\frac{2A}{\mu_1-\ell}+\big(S(0)+I(0)\big)\bigg).
\end{align*}
Choosing $\frac{1}{\zeta}\leq \frac{\mu_1 q\eta^q}{8\beta}(\mathcal{T}_0^s-1)\Big(\frac{2A}{\mu_1-\ell}+\big(S(0)+I(0)\big)\Big)^{-1}$. We thus obtain
\begin{align*} 
\underset{t\to +\infty}{\lim \sup}\frac{1}{t}\int^t_0\E(\mathbf{1}_{\Omega_4})du&\leq \frac{\mu_1 q\eta^q}{8\beta}(\mathcal{T}_0^s-1).
\end{align*}
According to (\ref{23}), one can derive that
\begin{align*}
\underset{t\to +\infty}{\lim \inf}\frac{1}{t}\int^t_0\E(\mathbf{1}_{\Sigma})du&\geq \underset{t\to +\infty}{\lim \inf}\frac{1}{t}\int^t_0\E(\mathbf{1}_{\Omega_1})du-\underset{t\to +\infty}{\lim \sup}\frac{1}{t}\int^t_0\E(\mathbf{1}_{\Omega_4})du \geq\frac{\mu_1  q\eta^q}{8\beta}\mathcal{T}_0^s>0\;\;\;\mbox{a.s.}
\end{align*}
Based on the above analysis,  we have determined a compact domain $\Sigma\subset\R^3_+$ such that
\begin{align}
\underset{t\to +\infty}{\lim \inf}\frac{1}{t}\int^t_0\P(u,(S(0),I(0),R(0)),\Sigma)du\geq \frac{\mu_1  q\eta^q}{8\beta}(\mathcal{T}_0^s-1)>0.
\label{end}
\end{align}
By (\ref{end}), we show that (\ref{imp}) is unverifiable. Applying similar arguments to those in \cite{11,12}, we show the uniqueness of the ergodic stationary distribution of our model (\ref{s2}).\\

 Now, we will prove that if $\mathcal{T}^s_0<1$, we have the exticntion of the disease. In view of (\ref{itoln}) and lamma \ref{lemmas} , we get that
 \begin{align*}
\underset{t\to \infty}{\lim \sup}\frac{1}{t}\ln\frac{ I(t)}{I(0)}&=\beta\underset{t\to \infty}{\lim \sup}\int_0^t S(s)ds-\Big((\mu_2+\gamma)+\frac{\sigma_2^2}{2}+\int_Z \eta_2(u)-\ln(1+\eta_2(u))\nu(du)\Big)\\&\leq\beta\underset{t\to \infty}{\lim}\int_0^t \psi(s)ds-\Big((\mu_2+\gamma)+\frac{\sigma_2^2}{2}+\int_Z \eta_2(u)-\ln(1+\eta_2(u))\nu(du)\Big)\\
&=(\mu_2-\gamma)\Big(\mathcal{T}^s_0-1\Big)<0\hspace*{0.3cm}\mbox{a.s.}
\end{align*}
This completes the proof.
\end{proof}
\section{Example}\label{sec3}
In this section, we will validate our theoretical result with the help of numerical simulations taking parameters from the theoretical data mentioned in the table \ref{value}. We numerically simulate the solution of the system (\ref{s2}) with initial value $(S(0), I(0), R(0)) = (0.4, 0.3, 0.1)$. For the purpose of showing the effects of the perturbations on the disease dynamics, we have realized the simulation $15000$ times. 
\begin{center}
\begin{tabular}{llll}
\hline
Parameters  \hspace*{0.5cm}&Description  \hspace*{0.5cm}&Value  \\
\hline
$A$ & The recruitment rate & 0.09 \\
$\mu_1$ & The natural mortality rate & 0.05\\
$\beta$ & The transmission rate &0.06 \\
$\gamma$ & The recovered rate &0.01 \\
$\mu_2$ &The general mortality &0.09 \\
\hline
\end{tabular}
\captionof{table}{Some theoretical parameter values of the model (\ref{s2}).}
\label{value}
\end{center}
\begin{figure}[H]
\subcaptionbox{The left figure is the stationary distribution for S(t), the right picture is the stationary distribution I(t).}
{\includegraphics[width=3.3in]{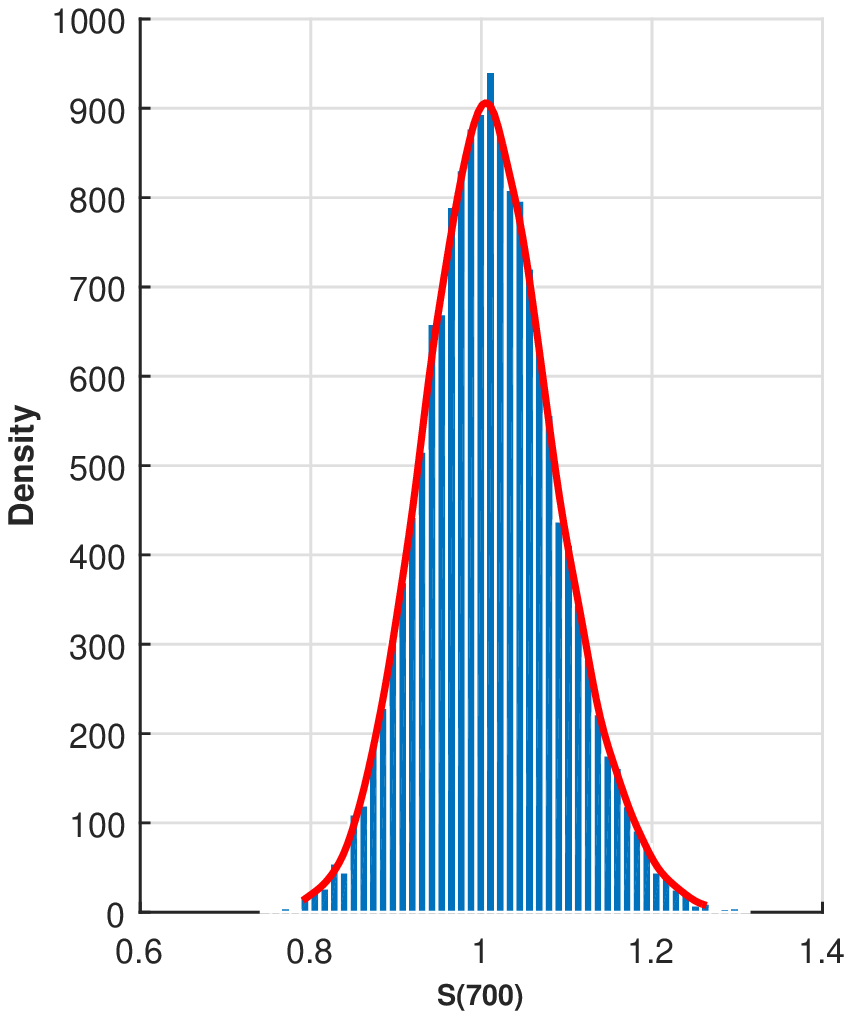} 
\includegraphics[width=3.3in]{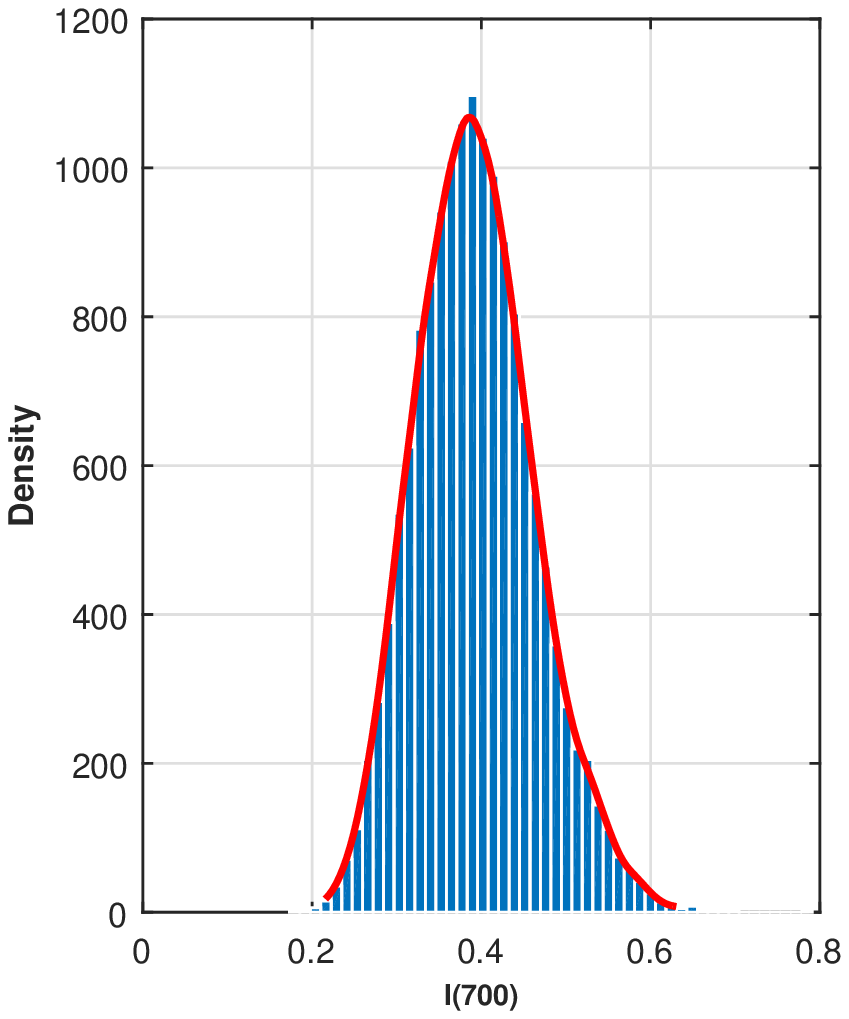}} 
\subcaptionbox{The left figure is the stationary distribution for R(t), the right picture is the trajectory of the solution.}
{\includegraphics[width=3.3in]{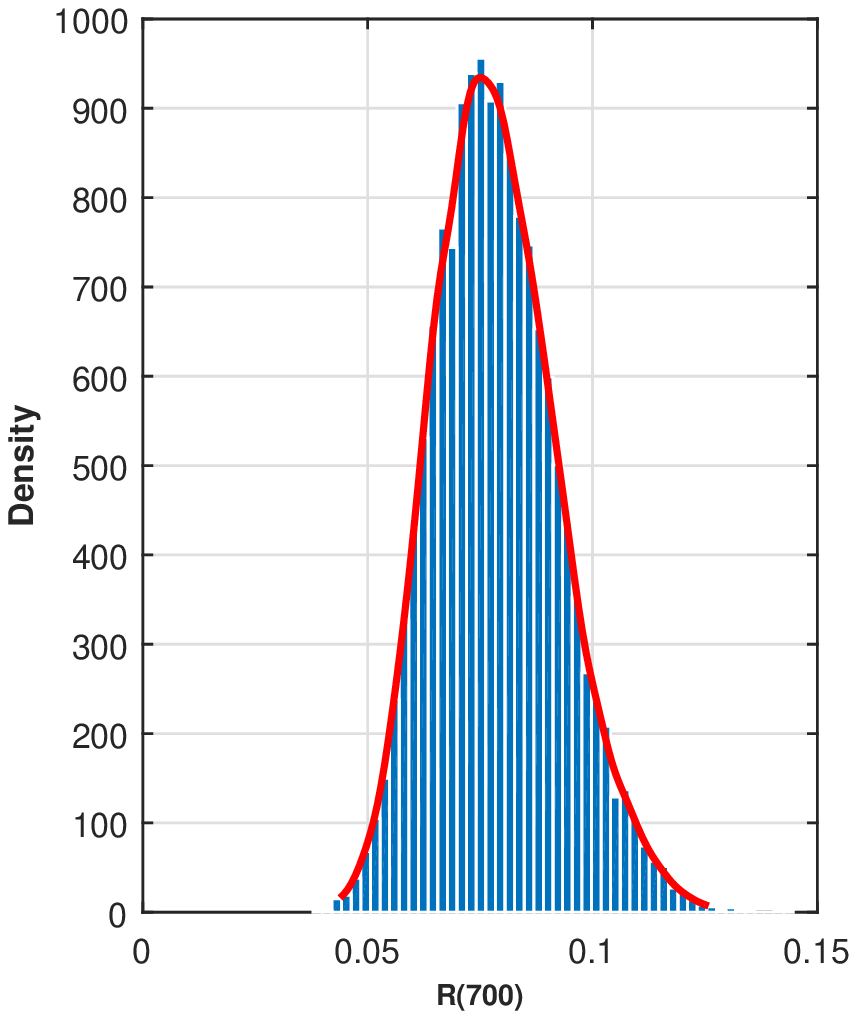} 
\includegraphics[width=3.2in]{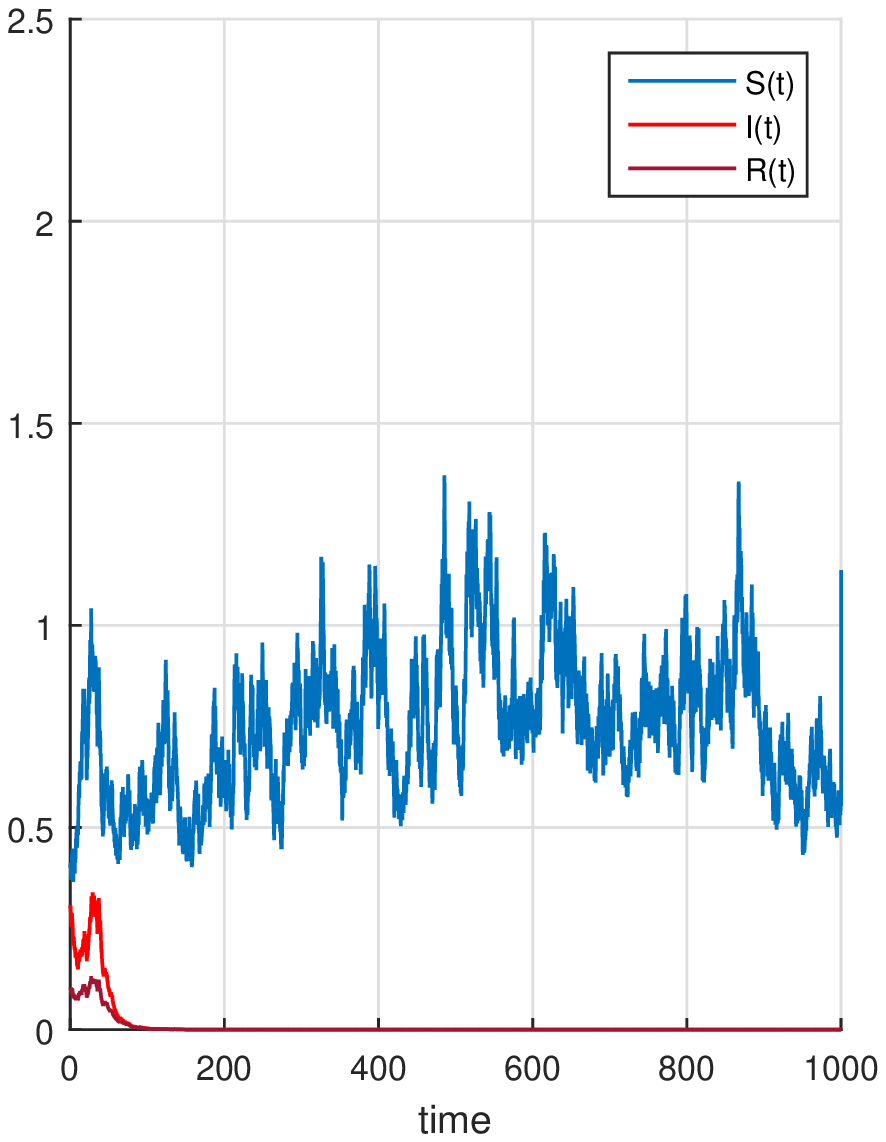} } 
\caption{The numerical illustration of obtained results in the theorem \ref{thm1}.}
\label{fig1}
\end{figure}
We have chosen the stochastic fluctuations intensities $\sigma_1=0.02$, $\sigma_2=0.08$ and $\sigma_3=0.01$. Furthermore, we assume that $\eta_1(u)=0.05$, $\eta_2(u)=0.02$, $\eta_3(u)=0.01$, $Z=(0,\infty)$ and $\nu(Z)=1$. Then, $\mathcal{T}^s_0=1.0460>1$. From figure \ref{fig1}, we show the existence of the unique stationary distributions for $S(t)$, $I(t)$ and $R(t)$ of model (\ref{s2}) at $t = 700$, where the smooth curves are the probability density functions of $S(t)$, $I(t)$ and $R(t)$, respectively (see figure \ref{fig1} (a) and (b)-left). Now, we choose $A=0.08$, Then, $\mathcal{T}^s_0=0.9260<1$. That is, $I(t)$ will tend to zero exponentially with probability one (see figure \ref{fig1} (b)-right).
\section{Conclusion}\label{sec4}
The dissemination of the epidemic diseases presents
a global issue that concerns decision-makers to elude deaths and deterioration of economies.  Many scientists are motivated to understand and suggest the ways for diminishing the epidemic dissemination. The first generation proposed the deterministic models that showed a lack of realism due to the neglecting of environmental perturbations.  Recent studies present a deep understanding of the process of outbreak diseases by taking into account their random aspect. This contribution presents new techniques to analyze the threshold of a stochastic SIR epidemic model with Lévy jumps.  We have based on the following new techniques:
\begin{enumerate}
\item The calculation of the temporary average of a solution of (\ref{s4}) instead of the classic method based on the explicit form of the stationary distribution in the model (\ref{s4}).
\item The use of Feller property and mutually exclusive possibilities lemma for proving the ergodicity of the model (\ref{s2}).
\end{enumerate}
According to the above techniques, our analysis leads to establish the threshold parameter for the existence of an ergodic stationary distribution and the extinction of the disease. 
 \section*{References} 
\bibliographystyle{ieeetr}
\bibliography{double}
\end{document}